\documentclass[11pt]{amsart}
\usepackage[utf8x]{inputenc}
\usepackage{amsfonts}
\usepackage{amssymb,epsfig}
\usepackage{amsmath}
\usepackage{delarray}
\usepackage{graphicx}
\newtheorem{thm}{Theorem}[section]

\newtheorem{cor}[thm]{Corollary}

\newtheorem{prop}[thm]{Proposition}

\newtheorem{defn}[thm]{Definition}
\theoremstyle{remark}
\newtheorem{remark}[thm]{Remark}
\numberwithin{equation}{section}
\newtheorem{example}[thm]{Example}

\newcommand{\cA}{\mathcal A}

\newcommand{\cC}{\mathcal C}
\newcommand{\cD}{\mathcal D}
\newcommand{\cF}{\mathcal F}
\newcommand{\cG}{\mathcal G}

\newcommand{\cK}{\mathcal K}
\newcommand{\cL}{\mathcal L}
\newcommand{\cM}{\mathcal M}
\newcommand{\cN}{\mathcal N}

\newcommand{\cU}{\mathcal U}

\newcommand{\cY}{\mathcal Y}

\newcommand{\bbR}{\mathbb R}
\newcommand{\bbT}{\mathbb T}

\newcommand{\bbZ}{\mathbb Z}

\newcommand{\rank}{{\rm rank\ }}

\begin{document}

\title{Action-Angle variables on Dirac manifolds}

\author{Nguyen Tien Zung}
\address{Institut de Mathématiques de Toulouse, UMR5219, Université Toulouse 3}
\email{tienzung.nguyen@math.univ-toulouse.fr}

\begin{abstract}{%
The main purpose of this paper is to show the existence of action-angle variables for integrable Hamiltonian
systems on Dirac manifolds under some natural regularity and compactness conditions, using the torus action approach. 
We show that the Liouville torus actions of general integrable dynamical systems have the structure-preserving
property with respect to any underlying geometric structure of the system, and deduce the existence of action-angle variables
from this property. We also discover co-affine structures on manifolds as a by-product of our
study of action-angle variables.
}\end{abstract}

\date{Version 1, 17 April 2012}

\maketitle

\section{Introduction}

Action-angle variables play a fundamental role in classical and quantum mechanics. They are
the starting point of the famous Kolmogorov--Arnold--Moser theory about the persistence of
quasi-preridicity of the motion integrable Hamiltonian systems under perturbations
(see, e.g., \cite{BroerSevruyk-KAM2010}). They are also the starting point of geometric 
quantization rules which go back to the works of Bohr, Sommefeld, 
Epstein and Einstein (see, e.g., \cite{BergiaNavarro-EinsteinQuantization2000}), 
and also of semi-classical quantization of integrable
Hamiltonian systems (see, e.g., \cite{San-Semiclassical2006}). The quasi-periodicity
of the movement of general proper integrable Hamiltonian systems in angle variables was 
discovered by Liouville \cite{Liouville-1855}. The first essentially complete proof of the theorem about the 
existence of action-angle variables near a Liouville torus on asymplectic manifold, 
which is often called Arnold--Liouville  theorem, is due to the astrophysicist Mineur 
\cite{Mineur-AA1935,Mineur-AA1935-37}, who was also motivated by the quantization problem. 

There have been generalizations of Arnold--Liouville--Mineur 
action-angle variables theorem to various contexts, including noncommutatively integrable systems
(see, e.g., \cite{Nekhoroshev-AA1972,MF-Noncommutative1978,DazordDelzant-AA1987}), systems on almost-symplectic
manifolds  \cite{FassoSansonetto-AA2007}, on contact manifolds 
(see, e.g., \cite{KhesinTaba-ContactIntegrable2010,Jovanovic_AAContact2011}), on Poisson manifolds 
\cite{LMV-AA2011}, and so on. Action-angle variables near singularities of integrable systems 
have also been studied (see, e.g. \cite{DufourMolino-AA1990, MirandaZung-NF2004,
San-Semiclassical2006,Zung-Integrable1996,Zung-Convergence2005,Zung-Torus2006}).  
But as far as we know, action-angle variables on general 
Dirac manifolds, or even general presymplectic manifolds, have not been studied in the literarture.
The aim of this paper is to remedy this situation, by extending the action-angle variables 
theorem to the case of integrable Hamiltonian systems on Dirac manifolds, which include systems on Poisson and
presymplectic manifolds as paricular cases. 

There are at least 2 reasons why we are interested in integrable Hamiltonian
on Dirac manifolds. The first is that Dirac structures appear naturally in systems with constraints.
For example, take an integrable Hamiltonian system on a Poisson manifold $(M,\Pi),$ and consider an invariant
submanifold
\begin{equation}
Q = Q_{c_1,\hdots, c_k}  = \{x \in M \ | \ F_1(x) = c_1,\hdots, F_k(x) = c_k\}
\end{equation}
where $(F_1,\hdots, F_k)$ is a partial family of commuting first integrals of the system.
Then, in general, under some regularity condition, $Q$ is a Dirac manifold (which is neither Poisson nor presymplectic)
on which the restricted system is still an integrable Hamiltonian system, and it still makes sense to talk about
action-angle variables on $Q$. In particular, we will make use of action-angle variables in 
a work in progress on the geometry of integrable dynamical systems on 3-manifolds.

The second reason is that, given an integrable dynamical system without any
a-priori Hamiltonian structure, it has more chance to become Hamiltonian with respect to a Dirac structure
than with respect to a symplectic or Poisson structure. This second reason may be important for applications
in fields like biology and economics, where due to ``forgotten'' variables even systems which are Hamiltonian
don't look Hamiltonian at the first sight at all. And since our setting is rather general, one can recover from the main results 
of this paper various  action-angle variables theorems in the literature.

The approach that we follow in this paper is a geometric approach based on the \emph{toric philosophy}: locally 
every dynamical system admits an intrinsic torus action which preserves anything which is preserved by the system.
(This torus action is a kind of double commutator). This toric philosophy is rather powerful. In particular, it allowed us to prove
the existence of convergent Poincaré-Birkhoff normalization for any analytic integrable system near a singular point
\cite{Zung-Convergence2005,Zung-Convergence2002}. In this paper, we will show that, for integrable systems near a regular 
level set, the role of this intrinsic torus action is played by nothing else than the Liouville torus action, i.e. 
the Liouville torus action preserves any underlying geometric structure
which is preserved by the system. In particular, for systems on Dirac manifolds, the Liouville torus action preserves
the Dirac structure. Using this property, one can construct relatively easily action-angle variables. We believe that,
even in the classical symplectic case, our proof is more conceptual and easier to understand and to generalize than
some existing proofs in the literature.

The organization of this paper is as follows: In Section 2 we recall some basic notions about Dirac structures and Hamiltonian
systems on Dirac manifolds, and show some  simple results about co-Lagrangian submanifolds on Dirac manifolds
(Theorem \ref{thm:co-Lagrangian1} and Theorem \ref{thm:co-Lagrangian2}), which are related to action-angle variables. 
In Section 3 we recall the notion of Liouville torus actions for general integrable dynamical systems, and show that
these Liouville torus actions have the structure-preserving property (Theorem \ref{thm:TorusPreservesStructure} and
Theorem \ref{thm:TorusPreservesStructureII}). Finally, Section 4 is about integrable Hamiltonian systems on Dirac manifolds,
where, using the results of Section 3, we show that if the system is Hamiltonian then
the Liouville torus action is also Hamiltonian
(Theorem \ref{thm:LiouvilleHamiltonianAction}), and deduce from this result the existence of action-angle variables,
both in the ``commutative'' case when the Liouville tori are Lagrangian (Theorem \ref{thm:fullAA} about full action-angle variables)
and in the  ``noncommutative'' case when the Liouville tori are isotropic (Theorem \ref{thm:partialAA} about partial action-angle variables). 
As a by-product of our study, we get a new kind of geometric structures on manifolds, called \emph{co-affine} structures
(see Subsection 4.3), which are induced from integrable Hamiltonian systems on presymplectic manifolds, 
and which seem to be very interesting by themselves.

\section{Dirac manifolds and co-Lagrangian submanifolds}

In this section, let us briefly recall some basic notions about Dirac manifolds and Hamiltonian systems on them 
(see, e.g., \cite{Bursztyn-Dirac2010,Courant-Dirac1990, Kosmann-Dirac2011},
Appendix A8 of \cite{DufourZung-PoissonBook}, and references therein). 
We will also write down some basic results about (co-)Lagrangian submanifolds of Dirac manifolds,
which are similar to Weinstein's results
on Lagrangian submanifolds of symplectic manifolds \cite{Weinstein-Symplectic1977},
and which are related to action-angle variables.

Dirac structures were  first used by Gelfand and Dorfman 
(see, e.g., \cite{GelfandDorfman-Hamiltonian1979,Dorfman-Dirac1987})
in the study of integrable systems, and were formalized by Weinstein and 
Courant in \cite{CourantWeinstein-Dirac1988,Courant-Dirac1990}
in terms of involutive isotropic subbundles of the ``big'' bundle $TM \oplus T^*M$. 
They generalize both (pre)symplectic and Poisson structures, and prove to be a 
convenient setting for dealing with  systems with constraints and reduction problems.

On the direct sum $TM \oplus T^*M$ of the tangent and the cotangent 
bundles of a smooth $n$-dmensional manifold $M$ there is a natural indefinite symmetric scalar
product of signature $(n,n)$ defined by the formula
\begin{equation}
 \langle (X_1,\alpha_1), (X_2,\alpha_2)\rangle = \frac{1}{2}(\langle \alpha_1,X_2\rangle + \langle \alpha_2,X_1\rangle )
\end{equation}
for sections $(X_1,\alpha_1), (X_2,\alpha_2) \in \Gamma (TM \oplus T^*M)$. A vector subbundle 
$\cD \subset TM \oplus T^*M$ is called {\bf isotropic} if the restriction of 
the indefinite scalar product to it is identically zero. 
On the space of smooth sections of $TM \oplus T^*M$ there is an operation, called the
\emph{Courant bracket}, defined by the formula
\begin{equation}
[(X_1,\alpha_1), (X_2,\alpha_2)] := ([X_1,X_2], \cL_{X_1}\alpha_2 - i_{X_2} d\alpha_1),
\end{equation}
where $\cL$ denotes the Lie derivative. A subbundle $\cD \subset TM \oplus T^*M$ is said to be {\bf closed} under the Courant
bracket if the bracket of any two sections of $\cD$ is again a section of $\cD.$

\begin{defn} A {\bf Dirac structure} on a $n$-dimensional manifold $M$ is an isotropic vector subbundle $\cD$ of rank $n$ of
$TM \oplus T^*M$ which is closed under the Courant bracket. If $\cD$ is a Dirac structure on $M$
then the couple $(M, \cD)$ is called a {\bf Dirac manifold}.
\end{defn}

We will denote the two natural projections $TM \oplus T^*M \to TM$ and $TM \oplus T^*M \to T^*M$ by $proj_{TM}$
and $proj_{T^*M}$ respectively. 

Grosso modo, a Dirac structure $\cD$ on a manifold $M$ 
is nothing but a singular foliation of $M$ by presymplectic leaves:
the singular {\bf characteristic distribution}  $\cC = proj_{TM} \cD$ of $\cD$ 
is integrable in the sense of Frobenius-Stefan-Sussmann due to the closedness condition.
On each leaf $S$ of the associated singular {\bf characteristic foliation} whose tangent distribution is $\cC$ 
there is an induced differential 2-form $\omega_S$  defined by the formula
\begin{equation}
 \omega_S (X,Y) = \langle \alpha_X, Y \rangle,
\end{equation}
where $X,Y \in \cC_x = T_xS$ and $\alpha_X$ is any element of $T^*_xM$ such that $(X,\alpha_X) \in \cD_x.$ 
Due to the closedness of $\cD,$ the 2-form $\omega_S$ is also closed, i.e.
$(S,\omega_S)$ is a presymplectic manifold. The Dirac structure $\cD$ is uniquely determined by its
characteristic foliation and the presymplectic forms on the leaves.

If  $proj_{TM}: \cD \to TM$ is bijective then the characteristic foliation consists of just 1 leaf, i.e. $M$ itself,
and $\cD$ is simply (the graph of) a presymplectic structure $\omega$ on $M$: 
$\cD = \{(X, X \lrcorner \omega) \ | \  X \in TM\}.$  On the other hand, if $proj_{TM}: \cD \to T^*M$ 
is bijective then $\cD$ is (the graph of) a Poisson structure on $M$, and the 2-forms $\omega_S$ are nondegenerate,
i.e. symplectic. However, in general, the ranks of the maps $proj_{T^M}: \cD \to TM$ and $proj_{T^M}: \cD \to TM^*$ may
be smaller than $n$, and may vary from point to point.

\begin{defn}
A Dirac structure $\cD$ on $M$ is called a {\bf regular Dirac structure} of {\bf bi-corank $(r,s)$} if there are two 
nonnegative integers $r,s$ such that $\forall x \in M$ we have
\begin{equation}
\dim (\cD(x) \cap T_xM) = n - \dim proj_{T^*M}\cD (x)  = r
\end{equation}
and
\begin{equation}
\dim (\cD(x) \cap T^*_xM) = n - \dim proj_{TM}\cD (x)  = s.
\end{equation}
\end{defn}

Even if the Dirac structure $\cD$ is non-regular, one can still talk about its bi-corank, defined to be the bi-corank of a
generic point in $M$ with respect to $\cD.$ For regular Dirac structures, we have the following analog of Darboux's theorem:

\begin{prop}[Darboux for regular Dirac] \label{thm:Darboux}
 Let $O$ be an arbitrary point of a $n$-manifold $M$ with a regular Dirac structure $\cD$ of bi-corank $(r,s)$.
Then $n - r - s = 2m$ for some $m \in \bbZ_+ $, and  there is a local coordinate
system $(x_1, \hdots, x_{2m}, y_1,\hdots,y_r,z_1, \hdots, z_s)$  in a neighborhood of $O$, 
such that the local characteristic foliation is of codimension $s$ and given by the local leaves 
\begin{equation}
 \{z_1 = const, \hdots, z_s = const\},
\end{equation}
and on each of these local leaves $S$ the presymplectic form $\omega_S$ is given by the formula
\begin{equation}
\omega_S = \sum_{i=1}^m dx_{2i-1} \wedge dx_{2i}.
\end{equation}
\end{prop}

The proof of the above proposition is essentially the same as the proof of the classical local Darboux normal form
for symplectic structures.

In particular, if $\cD$ is regular, then 
the {\bf kernel distribution} given by the kernels of the presymplectic forms is regular and integrable, and gives rise
to a foliation called the {\bf kernel foliation} of $\cD$. In local canonical coordinates given by Theorem \ref{thm:Darboux},
the kernel distribution is spanned by $(\frac{\partial}{\partial y_1}, \hdots, \frac{\partial}{\partial y_r}).$

\begin{example}
 Given a manifold $L$, a regular foliation $\cF$ on $L$, and a vector bundle $V$ over $L$, put $M = T^*\cF \oplus V$, where
$T^*\cF$ means the cotangent bundle of the foliation $\cF$ over $L$. Then $M$ admits the following regular Dirac structure $\cD$,
which will be called the {\bf canonical Dirac structure}: each leaf $S$ of the characteristic folitation is of the type
 $S = T^*N \oplus V_N = \pi^{-1} (N)$, where $N$ is a leaf of $\cF$ and $\pi: T^*\cF \oplus V \to L$ is the projection map,
and the presymplectic form on $S = T^*N \oplus V_N$ is the pull-back of the standard symplectic form on the cotangent bundle
$T^*N$ via the projection map $T^*N \oplus V_N \to T^*N.$ When $\cF$ consists of just one leaf $L$ and $V$ is trivial then
this canonical Dirac structure is the same as the (graph of the) standard symplectic structure on $T^*L.$
\end{example}

The notions of Hamiltonian vector fields  and Hamiltonian group actions can be naturally extended from the symplectic and
Poisson context to the Dirac context. In particular, we have:

\begin{defn}
 A vector field $X$ on a Dirac manifold $(M,\cD)$ is called a {\bf Hamiltonian vector field} if 
there is a function $H$, called a {\bf Hamiltonian function} of $X$, such that 
one of the following two equivalent conditions is satisfied:  \\
i) $(X,dH)$ is a section of $\cD$:
\begin{equation}
 (X, dH) \in \Gamma(\cD).
\end{equation}
ii) $X$ is tangent to the characteristic distribution and 
\begin{equation}
 X \lrcorner \omega_S = - d(H|_S)
\end{equation}
on every presymplectic leaf $(S,\omega_S)$ of it.
\end{defn}

\begin{prop}
 If $X$ is a Hamiltonian vector field of a Hamiltonian function $H$
on a Dirac manifold $(M,\cD),$ then $X$ preserves the Dirac structure $\cD,$ the  function $H$, and every leaf
of the characteristic foliation.
\end{prop}

\begin{defn} A function $f$ on $(M,\cD)$ is called a {\bf Casimir function} if $H$ is a Hamiltonian function of the trivial vector field,
i.e. $(0,df) \in \Gamma(\cD).$ A vector field $X$ on $(M,\cD)$ is called an {\bf isotropic vector field} if it is Hamiltonian with respect
to the trivial function, i.e. $(X,d0) \in \Gamma(\cD),$ or equivalently, $X$ lies in the kernel of the induced presymplectic forms. 
\end{defn}

Notice that if $X$ is a Hamiltonian vector field of Hamiltonian function $H$, $Y$ is an isotropic vector field, and $f$ is a Casimir
function, then $X+Y$ is also a Hamiltonian vector field of $H$, and $X$ is also a Hamiltonian vector field of $H + f$. Modulo
the isotropic vector fields and the Casimir functions, the correspondence between Hamiltonian vector fields and
Hamiltonian functions will become bijective.

Remark also that, unlike the Poisson case, not every function on a general Dirac manifold can be a Hamiltonian function
for some  Hamiltonian vector field. A necessary (and essentially sufficient) condition for a function $H$ to be a Hamiltonian
function is that the differential of $H$ must annulate the kernels of the induced presymplectic forms.

Another interesting feature of general Dirac structures  is that it is easier for a dynamical system
to become Hamiltonian with respect to a Dirac structure than with respect to a symplectic or 
Poisson structure, as the following example shows:

\begin{example} (See \cite{ZungMinh_2D2012}). A local 2-dimensional integrable vector field with a hyperbolic singularity
$X = h(x,y) (\frac{x}{a} \frac{\partial}{\partial x} - \frac{y}{b} \frac{\partial}{\partial y})$, where $a,b$ are
two coprime natural numbers  is not Hamiltonian with respect to any symplectic or Poisson
structure if $a+b \geq 3,$ but is Hamiltonian with respect to the presymplectic structure
$\omega = x^{a-1} y^{b-1}dx \wedge dy$. On the other hand, 
a local integrable vector field $X = h(y) y \frac{\partial}{\partial x}$ is not Hamiltonian with respect to any presymplectic
structure, but is Hamiltonian with respect to the Poisson structure $y \frac{\partial}{\partial x} \wedge \frac{\partial}{\partial y}$.
If an integrable vector field on a surface admits both of the above singularities  then it cannot
be Hamiltonian with respect to any presymplectic or Poisson structure, but  may  be Hamiltonian with respect to
a Dirac structure.
\end{example}

The theory of isotropic, coisotropic, and Lagrangian submanifolds can be naturally extended from the symplectic category to the Dirac
category. However, in the Dirac case, we will have to distinguish between the Lagrangian and the co-Lagrangian
submanifolds (which are the same thing in the symplectic case).

\begin{defn} Let $(M,\cD)$ be a Dirac manifold, where $\cD$ is a regular Dirac structure of bi-corank $(r,s)$. \\
i)  A submanifold $N$ of  $(M, \cD)$ is called {\bf isotropic} if it 
lies on a characteristic leaf $S$, and the pull-back of
the presymplectic form $\omega_S$ to $N$ is trivial. If, moreover, $N$ is of maximal dimension possible, i.e.
\begin{equation}
 \dim N =  \frac{1}{2} \rank \omega_S  + r = \frac{1}{2}(\dim M + r -s),
\end{equation}
then $N$ is called a {\bf Lagrangian submanifold}.
A foliation (or fibration) on $(M,\cD)$ is called {\bf Lagrangian} if its leaves (or fibers) are Lagrangian. \\
ii) A submanifold $L$ of  $(M,\cD)$ is called a {\bf co-Lagrangian submanifold} if 
\begin{equation}
\dim L = \frac{1}{2} (\dim M - r + s),
\end{equation}
and for every point $x \in L$  the tangent space $T_xL$ satisfies the following conditions : 
a) $T_xL + proj_ {TM}\cD(x) = T_xM;$ b) $T_xL \cap (T_xM \cap \cD(x)) = \{0\};$ 
c) $\omega_S|_{T_xL} = 0$ where $S$ is the characteristic leaf containing $x$.
\end{defn}

Observe that  if  $N$ is a Lagrangian submanifold then $T_xN$ contains the kernel of the presymplectic form at
$x$ for every $x \in N,$  this kernel distribution is regular and integrable in $N$, and $N$ is foliated by the
kernel foliation. If $L$ is a co-Lagrangian submanifold then $L$ is also foliated: the foliation on $L$ is the
intersection of the characteristic foliation of $(M,\cD)$ with $L$. Moreover, if 
$N$ is a Lagrangian submanifold and $L$ is a co-Lagrangian submanifold of $(M,\cD)$, 
then
\begin{equation}
 \dim N + \dim L = \dim M.
\end{equation}

The above definition of Lagrangian and co-Lagrangian submanifolds may differ from the other definitions
in the literature, but they are well-suited for our study of action-angle variables. In particular, it is easy to
see that any local Lagrangian foliation in a regular Dirac manifold admits a local co-Lagrangian section.
We also have the following analogs of some results of Weinstein \cite{Weinstein-Symplectic1977} 
about (co-)Lagrangian submanifolds:

\begin{thm}[Neighborhood of a co-Lagrangian submanifold] \label{thm:co-Lagrangian1}
Let $L$ be a co-Lagrangian submanifold of a regular
Dirac manifold $(M,\cD)$. Then there is a foliation $\cF$ on $L$, a vector bundle $V$ over $L$, and a Dirac
diffeomorphism from a neighborhood $(\cU(L), \cD)$ of $L$ to an open subset of 
$T^*\cF \oplus V$ equipped with the canonical Dirac structure, which sends $L$ to the zero section of $T^*\cF \oplus V.$ 
\end{thm}

\begin{proof}
 Let us first prove the above theorem in the Poisson case: $\cD = \{(\alpha\lrcorner \Pi, \alpha) \ | \ \alpha \in T^*M\}$
is the graph of a regular Poisson structure $\Pi$ on $M$. In this case, the bi-corank of $\cD$ is of the type $(0,s)$, and
the leaves of the characteristic foliation are symplectic. 

Denote by $\cF$ the foliation on $L$, which is the intersection of the characteristic foliation with $L$:
$T_x\cF = T_xN \cap \cC_xF$ for every $x \in N$, where $\cC$ is the characteristic distribution.  Then, via the
symplectic form on $\cC$,  the vector bundle $T^*\cF$ over $L$ is naturally isomophic to another vector bundle 
over $L$, whose fiber over $x \in L$ is the quotient space $\cC_x/T_x\cF.$ This latter bundle is also naturally 
isomorphic to the normal bundle of  $L$ in $M$. Due to these isomorphisms, there is a vector subbundle $E$ over $L$
of $\cC_L = \cup_{x \in L} \cC_x$, such that $\cC_L = T\cF \oplus E,$ and $E$
is Lagrangian, i.e. $E$ is isotropic with respect to the induced 
symplectic forms on $\cC$ and the rank of $E$ is half the rank of $\cC_L$.

At each point $x \in L,$ the set of germs of local Lagrangian submanifolds in $M$ which contain $x$ and which are tangent
to $E_x$ at $x$ is a contractible space. (By a local symplectomorphism from the characteristic leaf $S$ which contains  $x$
to $\bbT^*\bbR^m$ where $2m = \rank \omega_S$, this space of germs can be idientified with 
the space of germs of exact 1-forms  on $(\bbR^m,0)$ whose 1-jets vanish  at the origin). 
Due to this fact, there are no topological obstructions to the existence of a Lagrangian
foliation in a sufficiently  small neighborhood of $L$ which is tangent to $E_x$ at every point $x \in L$. Denote
by $\cN$ such a Lagrangian foliation. Identify $L$ with the zero section of $T^*\cF$. Then, similarly to the proof of  
uniqueness of marked symplectic realizations of  Poisson manifolds  (see Proposition 1.9.4 of \cite{DufourZung-PoissonBook}), 
one can show that there is a \emph{unique} Poisson isomorphism $\Phi$ from a neighborhood  of $N$ in $M$ to 
a neighborhood of $L$  of in $T^*\cF$, which is identity on $L$ and which sends the leaves of $N$ to the local fibers of
$T^*\cF$. $\Phi$ can be constructed as follows: 

Take a local function $F$ in the neighborhood of a point $x \in L$
in $M$, which is invariant on the leaves of the Lagrangian foliation $\cN$. Push $F$ to $T^*\cF$  by identifying $L$ with the zero
section of $T^*\cF$ and by making the function invariant on the fibers of $T^*\cF$. Denote the obtained local
function on $T^*\cF$ by $\tilde{F}$. Now extend the map $\Phi$ from $L$  (on which $\Phi$ is the identity map) to a 
neigborhood of $L$ by the flows of the Hamiltonian vector fields $X = X_F$ and $\tilde{X} = X_{\tilde{F}}$
of $F$ and $\tilde{F}$: if $y = \phi^t_X(z)$ where $z \in L$ and $\phi^t_X$ denotes the time-$t$ flow of $X$, then 
$\Phi(y) = \phi^t_{\tilde{X}} (z).$
One verifies easily that $\Phi$ is well-defined (i.e. it does not depend on the choice of the functions $F$), and is a required Poisson isomorphism.

Consider now the general regular Dirac case. Denote by $\cK$ the kernel foliation in a small tubular  neighborhood $\cU(L)$
of $L$ in $M$ in this case: the tangent space of $\cK$ at each point is the kernel of the induced presymplectic form at that point. 
Denote by $\cM \subset \cU(L)$ a submanifold which contains $N$ and which is transversal to the kernel foliation.
Then $\cM$ is Poisson submanifold of $(M,\cD)$.
Denote by $\pi_1: \cU(L) \to  \cM$ the projection map (whose preimages are the local leaves
of the lernel foliation).
 The Dirac structure $\cD$ in $\cU(L)$ is uniquely obtained from the
Poisson structure on $\cM$ by pulling back the symplectic 2-forms from the characteristic leaves of $\cM$ to the
characteristic leaves of   $\cU(L)$ via the projection map $\pi_1$
(so that they become presymplectic with the predescribed kernels). Denote by $V$ the vector bundle over $L$
which is the restriction of the kernel distribution to $L$.

According to the Poisson case of the theorem, there is a Poisson diffeomorphism from $\cM$ to a neighborhood of the zero
section in $T^*\cF$.  Extend $\Phi$ to an arbitrary diffeomorphism $\hat{\Phi}$ from $\cU(L)$
to a neighborhood of the zero section in $T^*\cF \oplus V$ which is fiber-preserving 
in the sense  that $\pi_2 \circ \hat{\Phi} = \Phi \circ \pi_1$, where $\pi_2$ denotes the projection $T^*\cF \oplus V \to T^*\cF.$
Then $\hat{\Phi}$ is a  required Dirac diffeomorphism.
\end{proof}

\begin{thm}[Co-Lagrangian sections] \label{thm:co-Lagrangian2}
Let $\cF$ be a regular foliation on a manifold $L$, and $V$ be a vector bundle over $L$. Then a section $K$
of the vector bundle $T^*\cF  \oplus V$ equipped with the canonical Dirac structure is a co-Lagrangian
submanifold of $T^*\cF  \oplus V$  if and only if $L_1 = (\theta, v)$, where $\theta \in \Gamma (T^*\cF)$ 
with $d_\cF \theta = 0,$ and $v \in \Gamma(V)$ is arbitrary.
\end{thm}

\begin{proof}
 The proof is the same as in the symplectic case, when $V$ is trivial and $\cF$ consists of just 1 leaf, i.e. $L$ itself.
\end{proof}

\section{Liouville torus actions}

Let us recall the following natural notion of integrability of dynamical systems which are not necessarily Hamiltonian
(see, e.g., \cite{Bogoyavlenskij-Extended1998,Zung-Convergence2002,Zung-Torus2006}): 

A $n$-tuple $(X_1\hdots,X_p, F_1,\hdots, F_q)$, where $p \geq 1, q \geq 0, p+q = n,$ $X_i$ are vector fields on
a $n$-dimensional manifold $M$ and $F_j$ are functions on $M$, is called an {\bf integrable system} of {\bf type $(p,q)$}
on $M$ if it satisfies the following commutativity and non-triviality conditions: \\
i) $[X_i,X_j] = 0 \ \ \forall i,j=1,\hdots,p$, \\
ii)  $X_i(F_j) = 0 \ \ \forall i \leq p, j \leq q$,  \\
iii) $X_1\wedge \hdots \wedge X_p \neq 0$ and $dF_1 \wedge \hdots \wedge dF_q \neq 0$ almost everywhere on $M$.

By a {\bf level set} of an integrable system $(X_1\hdots,X_p, F_1,\hdots, F_q)$ we mean a connected component $N$ of
a joint level set
\begin{equation}
 \{F_1 = const,\hdots, F_q = const\}.
\end{equation}
Notice that, by definition, the vector fields $X_1,\hdots, X_p$ are tangent to the level sets of the system.
We will say that the system $(X_1\hdots,X_p, F_1,\hdots, F_q)$ is {\bf regular} 
at $N$ if $X_1\wedge \hdots \wedge X_p \neq 0$ and $dF_1 \wedge \hdots \wedge dF_q \neq 0$  everywhere on $N$.

The following theorem about the existence of a system-preserving torus action near a compact regular level set of an
integrable system is essentially due to Liouville \cite{Liouville-1855}: 

\begin{thm}[Liouville's theorem] \label{thm:Liouville}
Assume that $(X_1,\hdots,X_p,F_1,\hdots,F_q)$ in an integrable system of type $(p,q)$ on a manifold $M$  which
is regular at a compact level set $N$.  Then in a tubular neighborhood $\cU(N)$ there is, up to automorphisms of $\bbT^p$, 
a unique free torus action
\begin{equation}
 \rho: \bbT^p \times \cU(N) \to \cU(N)
\end{equation}
which preserves the system (i.e. the action preserves each $X_i$ and each $F_j$)
and whose orbits are regular level sets of the system.  In particular, $N$ is diffeomorphic to  $\bbT^p$, 
and
\begin{equation}
 \cU(N) \cong \bbT^p \times B^q
\end{equation}
with periodic coordinates $\theta_1,\hdots,\theta_p (mod\ 1)$ on $\bbT^p$ and coordinates
$(z_1,\hdots, z_q)$ on a $q$-dimensional ball $B^q$, such that $F_1,\hdots, F_q$ depend only on the variables
$z_1,\hdots, z_q,$ and the vector fields $X_i$ are of the type
\begin{equation}
 X_i = \sum_{j=1}^p a_{ij}(z_1,\hdots,z_q) \frac{\partial}{\partial \theta_j}.
\end{equation}
\end{thm}

The proof of the above theorem is absolutely similar to the case of integrable Hamiltonian systems on symplectic manifolds,
see, e.g., \cite{Bogoyavlenskij-Extended1998,Zung-Torus2006}. It consists of 2 main points: 1) The map
$(F_1,\hdots,F_q): \cU(N) \to \bbR^q$ from a tubular neighborhood of $N$ to $\bbR^q$ is a 
topologically trivial fibration by the level
sets, due to the compactness of $N$ and the regularity of $(F_1,\hdots,F_q)$ (attention: if $(F_1,\hdots,F_q)$ is not 
regular at $N$ then this fibration may be non-trivial and may be twisted even if the level sets are smooth); 2) The vector fields
$X_1,\hdots,X_p$ generate a transitive action of $\bbR^p$ on the level sets near $N$, and the level sets are compact
and of dimension $p$, which imply that each level set is a $p$-dimensional compact quotient of $\bbR^p$, i.e. a torus.

The regular level sets in the above theorem are called 
{\bf Liouville tori}. We will also call the torus action in the above theorem the {\bf Liouville torus action}. 
Theorem \ref{thm:Liouville} shows that the  flow of the vector field $X=X_1$ of an integrable system
is quasi-periodic under some natural compactness and regularity conditions. This is the most fundamental geometrical
property of proper integrable dynamical systems. 

\begin{defn}
 An integrable system  $(X_1\hdots,X_p, F_1,\hdots, F_q)$  on a Dirac manifold $(M,\cD)$ is called
an {\bf integrable Dirac system} of type $(p,q)$ if the vector fields $X_1,\hdots, X_p$ preserve the Dirac structure.
\end{defn}

\begin{thm}[Liouville action preserves the Dirac structure]  \label{thm:TorusPreservesStructure}
If an integrable Dirac system  $(X_1\hdots,X_p, F_1,\hdots, F_q)$ 
on a Dirac manifold $(M,\cD)$ is regular at a compact level set $N$,
then the Liouville torus action in a tubular neighborhood of $N$ preserves the Dirac structure. 
\end{thm}

\begin{proof}
Let us first prove the above theorem for the Poisson case, i.e. when $\cD = \{(\alpha \lrcorner \Pi, \alpha) \ | \ \alpha \in T^*M\}$
is the graph of a Poisson tensor $\Pi$.

According to Liouville's Theorem \ref{thm:Liouville}, there is a coordinate system
\begin{equation}
(\theta_1 (mod\ 1), \hdots, \theta_p (mod\ 1), z_1,\hdots, z_q)
\end{equation}
 in $\cU(N)$, in which the vector fields $X_1, \hdots, X_p$ are of the form
\begin{equation}
X_i = \sum_{j=1}^p a_{ij}(z_1,\hdots, z_p)   \frac{\partial}{\partial \theta_j}.
\end{equation}
We will write the Poisson structure $\Pi$ as 
\begin{equation}
 \Pi = \sum_{i < j} f_{ij} \frac{\partial}{\partial \theta_i} \wedge \frac{\partial}{\partial \theta_j}
+ \sum_{i,j} g_{ij} \frac{\partial}{\partial \theta_i} \wedge \frac{\partial}{\partial z_j} +
\sum_{i < j} h_{ij} \frac{\partial}{\partial z_i} \wedge \frac{\partial}{\partial z_j}.
\end{equation}
The fact that $(X_1\hdots,X_p, F_1,\hdots, F_q)$ is a Dirac system with respect to $\Pi$ means that
$[X_k,\Pi] = 0$ for all $k=1,\hdots, p$ (where $[X_i,\Pi]$ denotes the Schouten bracket of $X_k$ with $\Pi$,
see, e.g., Chapter 1 of \cite{DufourZung-PoissonBook}). In other words, we have, for all $k \leq p$:
\begin{multline}
 \sum_{i < j} X_k(f_{ij}) \frac{\partial}{\partial \theta_i} \wedge \frac{\partial}{\partial \theta_j}
+ \sum_{i,j} X_k(g_{ij}) \frac{\partial}{\partial \theta_i} \wedge \frac{\partial}{\partial z_j}
+ \sum_{i,j} g_{ij} \frac{\partial}{\partial \theta_i} \wedge [X, \frac{\partial}{\partial z_j}] + \\
 \sum_{i < j} X_k(h_{ij}) \frac{\partial}{\partial z_i} \wedge \frac{\partial}{\partial z_j}
+ \sum_{i < j} h_{ij} [X_k,\frac{\partial}{\partial z_i}] \wedge \frac{\partial}{\partial z_j}
+ \sum_{i < j} h_{ij} \frac{\partial}{\partial z_i} \wedge [X_k, \frac{\partial}{\partial z_j}] = 0
\end{multline}
Notice that the coefficient of the term $\frac{\partial}{\partial z_i} \wedge \frac{\partial}{\partial z_j}$
in the above expression is $X_k(h_{ij})$, because
 $[X_k, \frac{\partial}{\partial z_i}] = - \sum_{j=1}^n \frac{\partial a_{kj}}{\partial z_i} \frac{\partial}{\partial \theta_j}$
does not contain the terms $\frac{\partial}{\partial z_j}.$ So we must have
\begin{equation}
X_k(h_{ij}) = 0 
\end{equation}
for all $k=1,\hdots,p$, which  implies that $h_{ij}$ is invariant on the level sets of the system in $\cU(N)$, i.e. 
$h_{ij}$ is invariant under the Liouville torus $\bbT^p$-action (for any indices $i < j$).

Denote by $\overline{f}_{ij} = \int_{\bbT^n} f_{ij} d\theta_1 \hdots d \theta_n$ (resp. $\overline{g}_{ij}, \overline{h}_{ij}$) the average of 
$f_{ij}$ (resp. $g_{ij}, h_{ij}$) with  respect to the Liouville torus acction. Since $h_{ij}$ is $\bbT^p$-invariant,
we have $\overline{h}_{ij} = h_{ij}$. Since $X_k$ preserves $\Pi$ and commutes with the Liouville torus action, 
it also preserves 
\begin{equation}
 \overline \Pi = \sum_{i < j} \overline f_{ij} \frac{\partial}{\partial \theta_i} \wedge \frac{\partial}{\partial \theta_j}
+ \sum_{i,j} \overline g_{ij} \frac{\partial}{\partial \theta_i} \wedge \frac{\partial}{\partial z_j} +
\sum_{i < j} \overline h_{ij} \frac{\partial}{\partial z_i} \wedge \frac{\partial}{\partial z_j}.
\end{equation}
It implies that
\begin{multline}
0 = [X_k, \Pi - \overline \Pi] = 
[X_k, \sum_{i < j} \hat f_{ij} \frac{\partial}{\partial \theta_i} \wedge \frac{\partial}{\partial \theta_j}
+ \sum_{i,j} \hat g_{ij} \frac{\partial}{\partial \theta_i} \wedge \frac{\partial}{\partial z_j}] \\
 = \sum_{i < j} X_k (\hat f_{ij})  \frac{\partial}{\partial \theta_i} \wedge \frac{\partial}{\partial \theta_j}
+ \sum_{i,j} X_k (\hat g_{ij})  \frac{\partial}{\partial \theta_i} \wedge \frac{\partial}{\partial z_j}
+ \sum_{i,j} \hat g_{ij}\frac{\partial}{\partial \theta_i} \wedge [X, \frac{\partial}{\partial z_j}] ,
\end{multline}
where $\hat f_{ij} = f_{ij} - \overline f_{ij}$ and $\hat g_{ij} = g_{ij} - \overline g_{ij}.$ The coefficient
of $\frac{\partial}{\partial \theta_i} \wedge \frac{\partial}{\partial z_j}$ in the above expression is $X_k(\hat g_{ij}),$
so we must have $X_k(\hat g_{ij}) = 0$ (for any $i,j,k$), which implies that $\hat g_{ij}$ is $\bbT^p$-invariant,
which in turn implies that $\hat g_{ij} = 0$ and $g_{ij}$ is $\bbT^p$-invariant. Similarly, we also have that
$f_{ij}$ is $\bbT^p$-invariant. Thus the Poisson stucture $\Pi$ is invariant under the Liouville torus action, and
the theorem is proved in the Poisson case.

Let us now reduce the general Dirac case to the Poisson case. Assume that the kernel $K_x = \cD_x \cap T_xM$
has dimension $\dim K_x = r > 0$ at a point $x \in N$. 
Due to the invariance of $\cD$, and hence of the kernel
distribution, with respect to the vector fields $X_1,\hdots,X_p$, it is easy to see that one can choose $r$ 
1-forms $\alpha_1,\hdots, \alpha_r$ in $\cU(N)$, which have constant coefficients 
in the coordinates $(\theta_1 (mod\ 1), \hdots, \theta_p (mod\ 1), z_1,\hdots, z_q)$, 
and such that their restrictions to the kernel $K_y = \cD_y \cap T_yM$
span the dual space $K^*_y$ of $K_y$  for any $y \in N$. Put
\begin{equation}
 W = \cU(N) \times B^r \cong \bbT^p \times B^q \times B^r
\end{equation}
with additional coordinates $(w_1,\hdots,w_r)$ on $B_r.$  Construct the following Dirac structure $\cD^W$
on $W$: each characteristic leaf $S^W$ of $\cD^W$ is the pull-back of a characteristic leaf  $S$ of $\cD$ 
via the projection map $\pi: W \to \cU(N)$:  $S^W = \pi^{-1}(S)$, and the presymplectic form $\omega_{S^W}$
on $S^W$ is given by the formula:
\begin{equation}
 \omega_{S^W} = \pi^*\omega + \sum_{i=1}^r \alpha_i \wedge dw_i.
\end{equation}
By construction, $ \omega_{S^W}$ is actually a symplectic form, i.e. $\cD^W$ is a Poisson structure.
Lift the vector fields $X_i$ from $\cU(N)$ to $W$ in a trivial way, by keeping the same formula
\begin{equation}
X_i = \sum_{j=1}^p a_{ij}(z_1,\hdots, z_p)   \frac{\partial}{\partial \theta_j}
\end{equation}
for them in the coordinates  $(\theta_1 (mod\ 1), \hdots, \theta_p (mod\ 1), z_1,\hdots, z_q,w_1,\hdots,w_r).$
Lift the Liouville torus $\bbT^p$-action from $\cU(N)$ to $W$ in the same way. Then we 
have an integrable system $(X_1,\hdots,X_p,F_1,\hdots,F_q,w_1,\hdots,w_r)$ of type $(p,q+r)$ on the Poisson
manifold $(W,\cD^W)$, which is regular at $N$. Applying the proved result in the Poisson case, we obtain
that $\cD^W$ is invariant with respect to the Liouville $\bbT^p$-action, which implies that $\cD$ is also
invariant with respect to the Liouville torus action. The theorem is proved.
\end{proof}

\begin{remark}
The Dirac structure in Theorem \ref{thm:TorusPreservesStructure} can be non-regular in the neighborhood of $N$.
In particular, the rank of the kernel distribution $\cD \cap TM$ is constant on each Liouville torus but may vary
from torus to torus. The vector fields $X_i$ in Theorem \ref{thm:TorusPreservesStructure} are not necessarily Hamiltonian,
and the Liouville torus $N$ is not necessarily isotropic in general. 
\end{remark}

Theorem \ref{thm:TorusPreservesStructure} agrees with the general philosophy about toric degree and 
intrinsic torus actions associated to vector fields: ``anything'' preserved by a vector field is also preserved by its 
intrinsic associated torus action, see, e.g., \cite{Zung-Convergence2002,Zung-Convergence2005,Zung-Torus2006}.
This philosophy also leads to the following result:

 \begin{thm}  \label{thm:TorusPreservesStructureII}
If an integrable system  $(X_1\hdots,X_p, F_1,\hdots, F_q)$ on a manifold $M$
 is regular at a compact level set $N$ and preserves a tensor field $\cG \in \Gamma (\otimes^kTM \otimes^h T^*M),$
then its Liouville torus action in a tubular neighborhood of $N$ also preserves the tensor field $\cG$. 
\end{thm}

\begin{remark}
In the above theorem, the tensor field $\cG$ can be anything: a vector field (an infinitesimal generator of a Lie group action),
a contact 1-form, a volume form, a metric, a Nambu structure, etc. So if the system preserves any
such structure then the associated Liouville torus action also preserves the same structure. The cases of Poisson  and
presymplectic structures  are already covered by Theorem \ref{thm:TorusPreservesStructure}.
\end{remark}

\begin{proof}
The proof is similar to the Poisson case of Theorem \ref{thm:TorusPreservesStructure}. Fix a canonical coordinate
system  $(\theta_1 (mod\ 1), \hdots, \theta_p (mod\ 1), z_1,\hdots, z_q)$ in a tubular neighborhood
$\cU(N)$ of $N$ as given by Theorem \ref{thm:Liouville}. We will make a filtration of the space 
$\Gamma (\otimes^kTM \otimes^h T^*M)$ of  tensor fields of contravariant order $k$  and contravariant order $h$ as follows:

The subspace $T^{h,k}_s$ consists of sections of  $\otimes^kTM \otimes^h T^*M$ whose expression in the coordinates
$(\theta_1 (mod\ 1), \hdots, \theta_p (mod\ 1), z_1,\hdots, z_q,w_1,\hdots,w_r)$
contains only  terms of the type
\begin{equation}
\frac{\partial}{\partial \theta_{i_1}} \otimes \hdots \otimes \frac{\partial}{\partial \theta_{i_a}} \otimes
\frac{\partial}{\partial z_{j_1}} \otimes \hdots \otimes \frac{\partial}{\partial z_{i_b}} \otimes
d\theta_{i'_1} \otimes \hdots \otimes d\theta_{i'_c} \otimes
dz_{j'_1} \otimes \hdots \otimes dz_{j'_d}
 \end{equation}
with $b+c \leq s$. For example,
\begin{equation} 
 T^{h,k}_0 =  \left\{ \sum_{i,j'} f_{i,j'}\frac{\partial}{\partial \theta_{i_1}} \otimes \hdots \otimes \frac{\partial}{\partial \theta_{i_h}} \otimes
dz_{j'_1} \otimes \hdots \otimes dz_{j'_k} \right\} .
\end{equation}
Put $T^{h,k}_{-1}= \{0\}$.  It is clear that
\begin{equation}
 \{0\} = T^{h,k}_{-1} \subset T^{h,k}_ 0 \subset T^{h,k}_ 1 
\subset \hdots \subset T^{h,k}_{h+k} = \Gamma (\otimes^kTM \otimes^h T^*M).
\end{equation}
It is also clear that the above filtration is stable under the Lie derivative of the vector fields $X_1,\hdots, X_p$, i.e.
we have
\begin{equation}
\cL_{X_{\alpha}}\Lambda \in T^{h,k}_s \ \ \forall  s=0,\hdots,k+h, \ \forall \Lambda \in T^{h,k}_s,\ \forall \alpha = 1,\hdots,p.
\end{equation}

Since $\cL_{X_\alpha} \cG = 0$ for all $\alpha=1,\hdots,p$ by our hypothesis, and the Liouville torus action
commutes with the vector fields $X_\alpha$,  we also have that $\cL_{X_\alpha} \overline \cG = 0,$ 
where the overline means the average of a tensor with respect to the Liouville torus action.
Thus we also have
\begin{equation}
\cL_X\hat {\cG} = 0,
\end{equation}
where 
\begin{equation}
 \hat \cG = \cG - \overline \cG
\end{equation}
has average equal to 0.

Similarly to the proof of the Poisson case of Theorem \ref{thm:TorusPreservesStructure}, the equalities
\begin{equation}
\cL_{X_\alpha} \hat \cG = 0 \ \forall \alpha =1,\hdots, p
\end{equation}
imply that the coefficients of $\hat \cG$ of the terms which are not in $T^{h,k}_{h+k-1}$, i.e. the terms of the type
\begin{equation}
\frac{\partial}{\partial z_{j_1}} \otimes \hdots \otimes \frac{\partial}{\partial z_{i_h}} \otimes
d\theta_{i'_1} \otimes \hdots \otimes d\theta_{i'_k}, 
 \end{equation}
are  invariant with respect to the vector fields $X_\alpha$. It means that these coefficient functions are invariant
with respect to the Liouville torus action.
But any $\bbT^p$-invariant function with average 0 is a trivial function, so in fact $\hat \cG$ does not contain any term
outside of $T^{h,k}_{h+k-1},$ i.e. we have:
\begin{equation}
\hat \cG \in  T^{h,k}_{h+k-1}.
\end{equation}
By the same arguments, one can verify that if $\hat \cG \in  T^{h,k}_s$ with $s \geq 0$ 
then in fact $\hat \cG \in  T^{h,k}_{s-1}.$ So by induction we have $\hat \cG = 0,$ i.e. $\cG = \overline \cG$ is invariant
with respect to the Liouville torus action.
\end{proof}

\section{Action-angle variables}

\subsection{Integrable Hamiltonian systems on Dirac manifolds} 
Liouville torus actions give rise to periodic coordinates on regular level sets of integrable systems. 
These periodic coordinates may be called \emph{angle variables}. They exist for any proper regular integrable system, 
without the need of any additional underlying structure. However, in order to get \emph{action variablles}, 
we will need the Hamiltonianity of the system. The word ``action'' itself means a Hamiltonian function or 
a momentum map which generates an associated Hamiltonian action. So in this section we will restrict our attention
to Hamiltonian systems.

\begin{defn}
 An integrable system  $(X_1\hdots,X_p, F_1,\hdots, F_q)$  on a Dirac manifold $(M,\cD)$ is called
an {\bf integrable Hamiltonian system} of type $(p,q)$ if the vector fields $X_1,\hdots, X_p$ are Hamiltonian,
i.e. there are Hamiltonian functions $H_1,\hdots, H_p$ such that $(X_i, dH_i) \in \Gamma(\cD)$ for $i=1,\hdots,p.$
\end{defn} 

Of course, an integrable Hamiltonian system on a Dirac manifold is 
also an integrable Dirac system, but the converse is not true. The above notion of integrable Hamiltonian systems
on Dirac manifolds generalizes the classical notion of integrability à la Liouville for Hamiltonian systems on
symplectic manifolds,  Mischenko-Fomenko's and Nekhoroshev's notions of noncommutative or generalized
integrability of  Hamiltonian systems
\cite{MF-Noncommutative1978,Nekhoroshev-AA1972} and some other notions of 
Hamiltonian (super)integrability as well.

\begin{prop} \label{prop:LiouvilleToriIsotropic}
 Let $(X_1\hdots,X_p, F_1,\hdots, F_q)$ be an integrable Hamiltonian system with corresponding Hamiltonian
functions $H_1,\hdots, H_p$ on a Dirac manifold $(M,\cD)$, and assume that this system is regular at a compact
level set $N$. Then we have: \\
i) The functions $H_1,\hdots, H_p$ are invariant on the Liouville tori in a tubular neighborhood $\cU(N)$
of $N.$ \\
ii)  The Liouville tori in $\cU(N)$ are isotropic. \\
iii) The functions $H_1,\hdots, H_p$ commute with each other in $\cU(N)$, i.e. their Poisson brackets vanish:
$\{H_i, H_j\} = 0.$
\end{prop}

\begin{proof}
 Recall that, similarly to the case of Poisson manifolds, if $H$ and $F$ are two Hamiltonian functions on a Dirac manifold
with two corresponding Hamiltonian vector fields $X_H$ and $X_F$, 
then their Poisson bracket $\{H,F\} := X_H(F) = - X_F(H) = \omega_S(X_H,X_F)$ (where $\omega_S$ denotes the
induced presymplectic forms) is again a Hamiltonian function whose associated 
Hamiltonian vector fields are equal to $[X_H,X_F]$ plus isotropic vector fields.

Since $[X_i,X_j] = 0$ (for any $i,j \leq p$) we have that $X_i(H_j) = \{H_i,H_j\}$ is a Casimir function. In particular,
$X_i(H_j)$ is invariant on the Liouville tori near $N$, because the Liouville tori belong to the characteristic leaves
(because the tangent bundle of the Liouville tori are spanned by the Hamiltonian vector fields $X_1\hdots,X_p$
which are tangent to the characteristic distribution). But the average of $X_i(H_j)$ on each Liouville torus is 0 due to the
quasi-periodic nature of $X_i$, so $X_i(H_j) = 0$ on each Liouville torus, i.e. we have
\begin{equation}
 X_i(H_j)   = \{H_i,H_j\} = 0 \ \text{in} \ \cU(N) \ \forall i=1,\hdots,p, 
\end{equation}
which implies that $H_j$ is invariant on the Liouville tori for all $j =1,\hdots, p.$

The fact that the Liouville tori are isotropic follows from the equation
$\omega_S(X_i,X_j) = \{H_i,H_j\} = 0$ and the fact that  the vector fields $X_1,\hdots,X_p$
span the tangent bundles of the Liouville tori.
\end{proof}

\subsection{Action functions}

\begin{thm}[Liouville action is Hamiltonian]  \label{thm:LiouvilleHamiltonianAction}
 Let $(X_1\hdots,X_p, F_1,\hdots, F_q)$ be an integrable Hamiltonian system,  which is regular at a compact
level set $N$, on a Dirac manifold $(M,\cD)$.  Assume moreover that one of the following two additional
conditions is satisfied: \\
i) The dimension $\dim (Span_\bbR(X_1(x),\hdots,X_p(x)) \cap (T_xM \cap \cD_x))$ 
is constant in a neighborhood of $N$. \\
ii) The characteristic foliation is regular 
in a neighborhood of $N$. \\
Then the Liouville torus action of the system is a Hamiltonian torus action (i.e. its generators are Hamiltonian) 
in a neighborhood $\cU(N)$ of $N$.
\end{thm}

In particular, if $\cD$ is a Poisson structure then the condition i) holds, and if $\cD$ is a presymplectic structure
then the condition ii) holds, and the theorem is valid in both cases. We don't know if the above theorem is still
true in the ``singular'' case when both of the above two conditions are false or not: we have not been able to produce
a proof nor a counter-example.

\begin{proof}
Let us first prove the theorem under condition ii), i.e. the characteristic foliation if regular. 

Fix a tubular neighborhood $\cU(N) \cong \bbT^p \times B^q$ with a coordinate system
$(\theta_1 (mod\ 1), \hdots, \theta_p (mod\ 1), z_1,\hdots, z_q)$ in which the vector fields $X_1,\hdots, X_p$ 
are constant on Liouville tori, as given by Theorem \ref{thm:Liouville}.
What we need to show is that $\frac{\partial}{\partial \theta_1}$ is a Hamiltonian vector field.
(Then, by similar arguments, the vector fields $\frac{\partial}{\partial \theta_2}, \hdots, 
\frac{\partial}{\partial \theta_p}$ are also Hamiltonian, so the Liouville torus action is Hamiltonian).
We can write
\begin{equation}
\frac{\partial}{\partial \theta_1}  = \sum_{i=1}^p b_i X_i,
\end{equation}
where the functions $b_i$ are invariant on the Liouville tori. Put
\begin{equation}
 \beta = \sum_{i=1}^p b_i dH_i.
\end{equation}
Then $(\frac{\partial}{\partial \theta_1}, \beta) = \sum_{i=1}^p b_i (X_i, dH_i) \in \Gamma(\cD).$
Since $\frac{\partial}{\partial \theta_1}$ preserves the Dirac structure, it also preserves the presymplectic
structure $\omega_S$ of each characteristic leaf $S$, and therefore
$d\beta|_S = d({\frac{\partial}{\partial \theta_1}}\lrcorner \omega_S) = 
\cL_{\frac{\partial}{\partial \theta_1}}\omega_S = 0$, i.e. the restriction of $\beta$ to each
characteristic leaf  is closed. Notice that, by condition ii) and Proposition \ref{prop:LiouvilleToriIsotropic}, each
characteristic leaf in $\cU(N)$ is a trivial fibration by Liouville tori over a disk. The 1-form $\beta$ is
not only closed, but actually exact, on each characteristic leaf, because its pull-back to each Liouville
torus is trivial by construction and by Proposition \ref{prop:LiouvilleToriIsotropic}.  

We can define a Hamiltonian function $A_1$ associated to $\frac{\partial}{\partial \theta_1}$ as follows:
Fix a point $x_0 \in N$, and let $D$ be a small disk containing $x_0$ which is transversal to the characteristic
foliation. Let  $H_1,\hdots, H_p$ be arbitrary Hamiltonian functions associated to $X_1,\hdots,X_p.$ 
For each $y \in \cU(N)$, denote by $y_0$ the  intersection point of the characteristic 
leaf through $y$ in $\cU(N)$ with $D$, and define
\begin{equation} \label{eqn:ActionFunctionIntegral}
A_1(y) = \int_{y_0}^y \beta,
\end{equation}
where the above integral means the integral of $\beta$ over a path on a characteristic leaf from $y_0$ to $y$.
The function  $A_1(y)$ is well defined, i.e. single-valued and does not depend on the choice of the path,
because of the exactness of $\beta$ on the characteristic leaves. 
It is also obvious that $dA_1 = \beta,$ i.e. $A_1$ is a Hamiltonian
function of $\frac{\partial}{\partial \theta_1}.$

Let us now assume that condition ii) fails, but condition i) holds, i.e.
$d = \dim (Span_\bbR(X_1(x),\hdots,X_p(x)) \cap (T_xM \cap \cD_x))$ is a constant on $\cU(N).$
Without loss of generality, we can assume that $X_1(x),\hdots,X_{p-d}(x)$ are linearly independent
modulo $Span_\bbR(X_1(x),\hdots,X_p(x)) \cap (T_xM \cap \cD_x))$ for any $x \in \cU(N)$. It implies
that $dH_1 \wedge \hdots \wedge dH_{p-d}(x) \neq 0$ everywhere in $\cU(N).$  By the inverse function
theorem, there exists a disk $D$ which intersects the characteristic leaf $S \ni x_0$ transversally at $x_0,$
and such that the functions $H_1,\hdots,H_{p-d}$ are invariant on $D$. 

Define the action function $A_1$ by the same Formula \eqref{eqn:ActionFunctionIntegral} as above, with
$y_0 \in D$. Since
the characteristic foliation in $\cU(N)$ is singular, a general characteristic leaf in $\cU(N)$ can intersect $D$ at 
a submanifold instead of just a point. In order to show that $A_1$ is well-defined, we have to check that
if $\gamma$ is an arbitrary oriented curve lying on the intersection of a characteristic leaf $S$ with the disk $D$,
then we have $\int_\gamma \beta = 0.$ But it is the case, because the pull-back of $dH_i$ to $\gamma$ is trivial
for all $i =1,\hdots,p$ by construction. Thus $A_1$ is a well-defined single-valued Hamiltonian function
of $\frac{\partial}{\partial \theta_1}$, and the theorem is proved.
\end{proof}

The Hamiltonian functions $A_1,\hdots,A_p$ of the generators $\frac{\partial}{\partial \theta_1}, \hdots, 
\frac{\partial}{\partial \theta_p}$ of the Liouville torus action given in Theorem \ref{thm:LiouvilleHamiltonianAction}
will be called {\bf action functions} or {\bf action variables} of the integrable system.
Notice that the action functions are  determined by the system only up to Casimir functions and up to a choice
of the generators of the Liouville torus action (or in other words, a choice of the basis of the torus $\bbT^p)$

\begin{remark}
Another way to obtain action variables in the symplectic case is by 
the classical integral formula due to Mineur \cite{Mineur-AA1935-37}, Einstein, 
and othe physicists (see, e.g. \cite{BergiaNavarro-EinsteinQuantization2000}):
\begin{equation} \label{eqn:MineurIntegral}
 A_1 = \int_{\gamma_1} \alpha,
\end{equation}
where $\alpha$ is an 1-form such that $d\alpha|_S = \omega_S$, and $\gamma_1$ is the loop generated by the
vector field $\frac{\partial}{\partial \theta_1}$ on the Liouville torus (for each torus). 
But it is not easy to use  Formula \eqref{eqn:MineurIntegral}
on Dirac manifolds, because of the problem of existence and regularity of $\alpha$ in the Dirac case.  That's why
in the proof of Theorem \ref{thm:LiouvilleHamiltonianAction} we used Formula \eqref{eqn:ActionFunctionIntegral}
instead of Formula \eqref{eqn:MineurIntegral} for the action functions.
\end{remark}

\subsection{Co-affine structures}

Unlike the case of integrable (à la Liouville or noncommutatively integrable) Hamiltonian 
systems on symplectic or Poisson manifolds, action functions of an integrable Hamiltonian system 
on a presymplectic or Dirac manifold need not be functionally
independent in general. For example, consider the simplest case of a proper integrable Hamiltonian system of type (2,1)
on a presymplectic 3-manifold. Then we have a Liouville $\bbT^2$-action and 2 action functions on a
1-dimensional family of Liouville tori, so in this case the two action functions are functionally dependent.

If $\cD$ is a presymplectic structure, then there is no Casimir function, and the action functions are uniquely
determined by the system   up to an integral affine transformation, similarly to the symplectic case.
The functional dependence of action functions of integrable Hamiltonian systems on presymplectic manifolds 
creates a new kind of geometric structures, which may be called  \emph{co-affine} structures:

\begin{defn}
 A co-affine chart of order $p$ on a manifold $P$ is a chart on a ball $B \subset P$
together with a map $\cA: B \to \bbR^p$.  An {\bf (integral) co-affine structure} on a manifold $P$
is an atlas $P = \cup_i B_i$ of affine charts $(B_i \subset P, \cA_i: B \to \bbR^p)$ such that
for any two chart $B_i$ and $B_j$ there is an (integral) affine transformation $T_{ij} : \bbR^p \to \bbR^p$
such that $A_j = T_{ij} \circ A_i$ on the intersection $B_i \cap B_j.$
\end{defn}

\begin{cor} If an integrable Hamiltonian system of type $(p,q)$ on a presymplectic manifold
is regular at a compact level set $N$, then the space of Liouville tori in a tubular neighborhood
$\cU(N)$ of $N$ (every Liouville torus is considered as one point of this space) 
is naturally equipped with an integral  co-affine structure induced by the system.
\end{cor}

Observe that, similarly to Riemannian structures, co-affine structures have a lot of local invariants. 
In particular, one can  talk about the local convexity and the curvature of a co-affine structure.

\subsection{Full action-angle variables}

As was shown in Proposition \ref{prop:LiouvilleToriIsotropic}, Liouville tori of integrable Hamiltonian
systems are isotropic. As a consequence, their dimension satisifes the following inequality:
\begin{equation}
\dim N \leq \frac{1}{2} \rank \omega_S + r,
\end{equation}
where $r = \dim (\cD_x \cap T_xM)$ is the corank of $\omega$ on the characteristic leaf containg a
Liouville torus $N$. The dimension of $N$ is the number of angle variables, and also the  number of
action variables that we can have. In the optimal case, when the above inequality becomes equality, i.e.
$N$ is a Lagrangian submanifold, then we will say that we have a \emph{full set of action-angle variables}.
The word ``full'' means that the presymplectic form in this case can be completely described in terms
of action-angle variables. More precisely, we have: 

\begin{thm}[Full action-angle variables]  \label{thm:fullAA}
Let $(M,\cD)$ be a regular Dirac manifold of bi-corank $(r,s)$ and dimension $n = 2m + r+s$,
and let 
\begin{equation}
(X_1,\hdots,X_{m+r},F_1,\hdots,F_{m+s}) 
\end{equation}
 be an  integrable Hamiltonian system 
which is regular at a compact level set $N$ on $(M,\cD)$. Then the Liouville tori of the system
in a tubular neighborhood $\cU(N)$ are Lagrangian submanifolds of $(M,\cD)$, and there is
a coordinate system   
\begin{equation}
(\theta_1 (mod\ 1), \hdots, \theta_{m+r} (mod\ 1), z_1,\hdots, z_{m+s}) 
\end{equation}
on 
\begin{equation}
\cU(N) \cong \bbT^{m+r} \times B^{m+s},
\end{equation}
and action functions 
\begin{equation}
A_1 = z_1,\hdots, A_m = z_m, A_{m+1},\hdots, A_{m+r} 
\end{equation}
on  $\cU(N)$, such that the functions  $A_{m+1},\hdots, A_{m+r}$ 
depend only on the  coordinates $z_1,\hdots,z_{m+s}$, 
the characteristic leaves of $\cD$ in $\cU(N)$ are
\begin{equation} \label{eqn:RegularLeaves}
 S_{c_1,\hdots,c_s} = \{z_{m+1}= c_1,\hdots z_{m+s} = c_s\}
\end{equation}
and the presymplectic form $\omega_S$ on each leaf $S= S_{c_1,\hdots,c_s}$ is
\begin{equation} \label{eqn:AA-PresympmlecticFrom}
 \omega_S = (\sum_{i=1}^{m+r} dA_i \wedge d \theta_i)|_S
\end{equation}
\end{thm}

\begin{proof}
The fact that the Liouville tori are lagrangian is given by Proposition \ref{prop:LiouvilleToriIsotropic} and the
definition of Lagrangian submanifolds. Since the fibration by Liouville tori is Lagrangian, we can choose a
co-Lagrangian section $D$ of this fibration, and a  coordinate system 
$(\theta_1 (mod\ 1), \hdots, \theta_{m+r} (mod\ 1), z_1,\hdots, z_{m+s}) $ on $\cU(N)$
such that the leaves of regular characteristic foliation is given by Formula \eqref{eqn:RegularLeaves} and
the  functions $\theta_1, \hdots, \theta_{m+r}$ vanish on $D$, i.e. the co-Lagrangian disk $D$ is
given by the equation
\begin{equation}
D = \{\theta_1 = 0, \hdots, \theta_{m+r} = 0\}. 
\end{equation}

The existence of action variables $A_1,\hdots,A_{m+r}$ corresponding to vector fields 
$\frac{\partial}{\partial \theta_1}, \hdots,\frac{\partial}{\partial \theta_{m+r}}$ is given by 
Theorem \ref{thm:LiouvilleHamiltonianAction}. Without loss of generality, we can assume that
$TN$ is spanned by $\frac{\partial}{\partial \theta_1}, \hdots,\frac{\partial}{\partial \theta_{m}}$
and the kernel $K = \cD \cap TM.$ Then  $dA_1 \wedge \hdots \wedge dA_m|_S \neq 0$
everywhere in $\cU(N)$, i.e.  the functions $A_1,\hdots, A_m$ are functionally independent
on the symplectic leaves, but $dA_1 \wedge \hdots dA_m \wedge d A_{m+i}|_S = 0 \ \forall i=1,\hdots, r.$
It follows that we can put $z_1 = A_1,\hdots z_m = A_m,$ and choose $z_{m+1},\hdots, z_{m+s}$
to be Casimir functions. 

It remains to prove Formula \eqref{eqn:AA-PresympmlecticFrom}. 
By the invariance of verything with respect to the
Liouville torus action, it is enough to prove this formula at a point $x \in D.$ Without loss of generality,
we can assume that $\{x\} = N \cap D.$

If $X, Y \in T_xS$ are two vector fields tangent to the characteristic foliation at $x$ such that $X,Y \in T_xN$, then
$\omega_S(X,Y) = 0$ due to the isotropy of $N$, and $dA_i(X) = dA_i(Y) = 0$ for all $i=1,\hdots,{m+r},$ which implies
that $(\sum_{i=1}^{m+r} dA_i \wedge d \theta_i) (X,Y) = 0.$

If $X,Y \in T_xD \cap T_xS$ then $\omega_S(X,Y) = 0$ because $D$ is co-Lagrangian, and    
$(\sum_{i=1}^{m+r} dA_i \wedge d \theta_i) (X,Y) = 0$ because $d\theta_i(X) = d\theta_i(Y) = 0$ by construction.

If $X = \frac{\partial}{\partial \theta_j} \in T_xN$ and $Y  \in T_xD \cup T_xS$ then by construction we also have
$\omega_S(X,Y) = \omega(\frac{\partial}{\partial \theta_j}, Y) = - dA_j (Y) = dA_j \wedge d \theta_j  (X,Y) =
(\sum_{i=1}^{m+r} dA_i \wedge d \theta_i) (X,Y).$

Since any vector pair $(X,Y) \in (T_xD \cap T_xS)^2$ can be decomposed into a linear combination of the pairs of the above types, 
Formula \eqref{eqn:AA-PresympmlecticFrom} is proved.
\end{proof}

\begin{remark}
The above theorem is the analog in the Dirac setting of the action-angle variables theorem for Hamiltonian systems on symplectic
or Poisson manifolds which are integrable à la Liouville. 
In the symplectic case,  the fibers of a regular Lagrangian fibration  with compact fibers are automatically tori, but this fact is no longer
true in the Dirac case: due to the degeneracy of the presymplectic forms on characteristic leaves, 
one can have non-torus Lagrangian fibrations with compact fibers on Dirac manifolds. So on a Dirac manifold 
we need not just a Lagrangian fibration, but an integrable Hamiltonian system, in order to get action-angle variables.
\end{remark}

\subsection{Partial action-angle variables} For non-commutatively integrable Hamiltonian systems on symplectic
or Poisson manifolds, there are not enough action-angles variables  
to form a complete  coordinate system, but one can complete these variables by some additional coordinates to form
canonical coordinate systems \cite{Nekhoroshev-AA1972,MF-Noncommutative1978,LMV-AA2011}. The same is
also true in the Dirac setting, when the Liouville tori are isotropic but not Lagrangian:

\begin{thm}[Partial action-angle variables]  \label{thm:partialAA}
Let 
\begin{equation}
(X_1,\hdots,X_{p},F_1,\hdots,F_{q}) 
\end{equation}
 be an  integrable Hamiltonian system 
which is regular at a compact level set $N$ on a regular Dirac manifold $(M,\cD)$ of bi-corank $(r,s)$, such that
the distribution $TN \cap \cD$ is regular of rank $d$ ($0 \leq d \leq r$) 
in a small tubular neighborhood $\cU(N)$  of $N$
fibrated by Liouville tori.  Then there is a coordinate system   
\begin{equation}
(\theta_1 (mod\ 1), \hdots, \theta_{p} (mod\ 1), z_1,\hdots, z_{q}) 
\end{equation}
on 
\begin{equation}
\cU(N) \cong \bbT^{p} \times B^{q}
\end{equation}
and action functions 
\begin{equation}
A_1 = z_1,\hdots, A_{p-d} = z_{p-d}, A_{p-d+1},\hdots, A_{p} 
\end{equation}
on  $\cU(N)$, such that the functions  $A_{p-q+1},\hdots, A_{p}$ 
depend only on the  coordinates $z_1,\hdots,z_{p-d},z_{q-s+1},\hdots, z_{q}$, 
the characteristic leaves of $\cD$ in $\cU(N)$ are
\begin{equation} \label{eqn:RegularLeavesII}
 S_{c_1,\hdots,c_s} = \{z_{q-s+1}= c_1,\hdots, z_{q} = c_s\},
\end{equation}
and the presymplectic form $\omega_S$ on each leaf $S= S_{c_1,\hdots,c_s}$ is of the form
\begin{equation} \label{eqn:AA-PresympmlecticFromII}
 \omega_S = (\sum_{i=1}^{p} dA_i \wedge d \theta_i)|_S + \sum_{p-d<i<j\leq q-s-r+d} f_{ij} dz_i \wedge dz_j|_S .
\end{equation}
\end{thm}

\begin{proof}
 The proof is similar to the proof of Theorem \ref{thm:fullAA}. We can assume that $TN$ 
is spanned by $\frac{\partial}{\partial \theta_1}, \hdots,\frac{\partial}{\partial \theta_{p-d}}$
and  $TN \cap \cD.$ Then the action functions $A_1,\hdots, A_{p-d}$ are independent on the
characteristic leaves, while the remaining action functions $A_{p-d+1},\hdots, A_p$ are
functionally dependent of $A_1,\hdots, A_{p-d}$ on each characteristic leaf, i.e. we can write
$A_{p-d+1},\hdots, A_p$ as functions of $A_1,\hdots, A_{p-d}, z_{q-s+1}, \hdots, z_{q},$
where $z_{q-s+1}, \hdots, z_{q}$ are Casimir functions, and we can put $z_1 = A_1,
\hdots, z_{p-q} = A_{p-d}$.

Take a section $D$ (of dimension $q$) of the fibration by Liouville tori in $\cU(N)$, with the 
following property: the image of $D$ by the projection $proj: \cU(N) \to \cU(N)/\cK,$
where $\cU(N)/\cK$ denotes the Poisson manifold which is the quotient of $\cU(N)$ by the regular
kernel foliation, is coisotropic of codimension $p-d$, and moreover the intersection of the kernel
foliation with $\cD$ is a regular foliation of dimension $r-d$ on $D$. 
We can choose  the angle variables $\theta_1,\hdots, \theta_p$
so that they vanish on $D$.

The closed 1-forms  $d\theta_i$ do not annulate the kernel $TM \cap \cD$ in general. 
But they do annulate $D_x \cap \cD_x$ for any point $x \in D$ by construction. So for each
$x \in D,$ there are $p-d$ linear combinations $\sum_{j=1}^p c_{ij} d\theta_j$ ($i=1,\hdots, p-d$)
which are linearly independent and which annulate the kernel $T_xM \cap \cD_x.$ Hence there exist
vectors $Y_1(x),\hdots, Y_{p-d}(x) \in T_xM$ such that 
$(Y_i(x), \sum_{j=1}^p c_{ij} d\theta_j(x)) \in \cD_x$ for $i=1,\hdots,p-d,$ and these vectors 
$Y_1(x), \hdots,Y_{p-d}(x)$ are linearly independent modulo the kernel $K_x = \cD_x \cap T_xM$
(i.e. no non-trivial linear combination of these vectors lies in $K_x$). By the coisotropy property of
$D$ (or more precisely, of the projection of $D$ in $\cU(N)/\cK$), we can choose $Y_1(x),\hdots,
Y_{p-d}(x)$ so that they belong to $T_xD.$ One verifies directly that the distribution $\cY$ on $D$
given by $\cY_x = Span_\bbR(Y_1(x),\hdots,Y_{p-d}(x)) \oplus (T_xD \cap \cD_x)$ is an integrable
regular distribution of dimension $p+ r - 2d.$ Choose the $q-p+2d-r$ 
coordinates $z_{p-d+1}, \hdots, z_{q-s-r+d}$ in such a way that they are invariant on the 
Liouville tori, and also invariant with respect to the
distribution $\cY$. Choose $r-d$ additional coordinates $z_{q-s-r+d+1},\hdots, z_{q-s}$ such that
they are also invariant on the Liouville tori, and such that their diffenretials when restricted to
the $(r-d)$-dimensional space $T_xD \cap \cD_x$ form a basis of the dual space of that space for any
point $x \in D$. Finally, one verifies that $ \omega_S - (\sum_{i=1}^{p} dA_i \wedge d \theta_i)|_S$
can be expressed as $\sum_{p-d<i<j\leq q-s} f_{ij} dz_i \wedge dz_j|_S$, in a way similar 
to the end of the proof of Theorem \ref{thm:fullAA} 
\end{proof}

\subsection{Final remarks}

In this paper we studied the existence of action-angle variables for a \emph{regular} 
integrable Hamiltonian system on regular Dirac manifolds near a compact level set. 
There is a series of works on action-angle variables near singularities
for integrable Hamiltonian on sympletic manifolds, see, e.g.,
\cite{DufourMolino-AA1990, MirandaZung-NF2004,
San-Semiclassical2006,Zung-Integrable1996,Zung-Convergence2005,Zung-Torus2006}.
It  would be natural to extend them to the case of Dirac manifolds.
Other interesting problems to consider include global action-angle variables, K.A.M. theory, and
quantization of integrable Hamiltonian systems on Dirac manifolds, etc.


\vspace{0.5cm}


\begin{thebibliography}{10}
\baselineskip0.4cm
\parskip-0.1cm


\bibitem{BergiaNavarro-EinsteinQuantization2000}
S. Bergia, L. Navarro, {\it On the early history of Einstein's quantization rule of 1917},
Archives internationales d'histoire des sciences, Vol. 50 (2000), No. 145,  321--373.

\bibitem{Bogoyavlenskij-Extended1998}
O.I. Bogoyavlenskij, {\it Extended integrability and bi-hamiltonian systems}, Comm. Math. Phys. 196
(1998), no. 1, 19–51.

\bibitem{BroerSevruyk-KAM2010}
H.W. Broer, M. B. Sevryuk, {\it Chapter 6 -- KAM theory: quasi-periodicity of dynamical systems}, Hanbook of
Dynamical Systems, Vol. 3 (2010), 249--344.

\bibitem{Bursztyn-Dirac2010}
H. Bursztyn, {\it A  brief introduction to Dirac manifolds}, preprint arXiv 1112.5037 (2011).

\bibitem{Courant-Dirac1990}
T. Courant, {\it Dirac manifolds}, Trans Amer Math Soc, no. 319 (1990): 631-661

\bibitem{CourantWeinstein-Dirac1988}
T. Courant, A. Weinstein, {\it Beyond Poisson structures}, Séminaire sudrhodanien de géométrie
VIII. Travaux en Cours 27, Hermann, Paris (1988), 39–49.

\bibitem{DazordDelzant-AA1987}
P. Dazord and T. Delzant, {\it Le problème general des variables actions-angles}, J. Differential Geom. 26
(1987), no. 2, 223–251.

\bibitem{Dorfman-Dirac1987}
I. Ya. Dorfman, {\it Dirac structures of integrable evolution equations},
Phys. Lett. A 125 (1987), no. 5, 240–246.

\bibitem{DufourMolino-AA1990}
J.-P. Dufour, P. Molino, {\it Compactification d’action de $\bbR^n$ et variables action-angle
avec singularités}, MSRI Publ., Vol. 20 (1990) (Séminaire Sud-Rhodanien de Géométrie,
Berkeley, 1989, P. Dazord and A. Weinstein eds.), 151-167


\bibitem{DufourZung-PoissonBook}
J.P. Dufour, N.T. Zung, Poisson structures and their normal forms, Progress in Mathematics, Vol. 242, 2005.


\bibitem{FassoSansonetto-AA2007}
F. Fassò, N. Sansonetto, {\it Integrable almost-symplectic Hamiltonian systems},  
J. Math. Phys.  48  (2007),  no. 9, 092902, 13 pp.


\bibitem{GelfandDorfman-Hamiltonian1979}
I. M. Gelfand, I. Ya. Dorfman, {\it Hamiltonian operators and algebraic
structures related to them}, Funct. Anal. Appl. 13 (1979), 248–262.

\bibitem{Jovanovic_AAContact2011}
B. Jovanovic, {\it Noncommutative integrability and action-angle variables in contact
geometry}, preprint 2011.

\bibitem{KhesinTaba-ContactIntegrable2010}
B. Khesin and S. Tabachnikov, {\it Contact complete integrability}, Regular and Chaotic Dynamics, 15
(2010), no. 4-5, 504-520


\bibitem{Kosmann-Dirac2011}
Yvette Kosmann-Schwarzbach, {\it Dirac pairs}, preprint arxiv 1104.1378 (2011).

\bibitem{LMV-AA2011}
C. Laurent-Gengoux, E. Miranda, P. Vanhaecke,  
{\it Action-angle  coordinates for integrable systems on Poisson manifolds},  
Int. Math. Res. Notices  No. 8 (2011),  1839–1869. 


\bibitem{Liouville-1855}
J. Liouville, {\it Note sur l’intégration des équations differentielles de la dynamique}, présentée
au bureau des longitudes le 29 juin 1853, 
Journal de Mathématiques pures et appliquées 20 (1855), 137-138.

\bibitem{Mineur-AA1935}
H. Mineur, {\it Sur les systèmes mécaniques admettant n intégrales premières uniformes
et l’extension a ces systèmes de la méthode de quantification de Sommerfeld}, C. R.
Acad. Sci., Paris 200 (1935), 1571–1573.


\bibitem{Mineur-AA1935-37}
H. Mineur, {\it Sur les systèmes mécaniques dans lesquels figurent des paramètres fonctions
du temps. Etude des systèmes admettant n intégrales premières uniformes en involution. 
Extension a ces systèmes des conditions de quantification de Bohr--Sommerfeld},
Journal de l’Ecole Polytechnique III(143): 173–191 and 237–270, 1937.

\bibitem{MirandaZung-NF2004}
E. Miranda, Nguyen Tien Zung, {\it Equivariant normal form for 
nondegenerate singular orbits of integrable Hamiltonian systems}, 
Ann. Sci. École Norm. Sup. (4) 37 (2004), no. 6, 819–839.

\bibitem{MF-Noncommutative1978}
A. S. Mishchenko and A. T. Fomenko, {\it Generalized Liouville method of integration of Hamiltonian
systems}, Funct. Anal. Appl. 12 (1978), 113–121.

\bibitem{Nekhoroshev-AA1972}
N. N. Nekhoroshev, {\it Action-angle variables and their generalizations}, Trans. Moskow Math. Soc. 26 (1972), 180–198.



\bibitem{San-Semiclassical2006}
V\~u Ngoc San, {\it Symplectic techniques for semiclassical completely integrable systems}, 
in Topological Methods in the Theory of Integrable Systems,
Editors A.V. Bolsinov, A.T. Fomenko and A.A. Oshemkov, Cambridge Scientific Publications (2006).

\bibitem{Weinstein-Symplectic1977}
A. Weinstein, Lectures on symplectic manifolds, Conference Board of the Mathematical Sciences, No. 29 (1977).

\bibitem{Zung-Integrable1996}
Nguyen Tien Zung, {\it Symplectic topology of integrable Hamiltonian systems. 
I. Arnold-Liouville with singularities}, Compositio Math. 101 (1996), no. 2, 179-215.


\bibitem{Zung-Convergence2002}
Nguyen Tien Zung, {\it Convergence versus integrability in Poincaré-Dulac normal form}, 
Math. Res. Lett. 9 (2002), no. 2-3, 217-228.
 
\bibitem{Zung-Convergence2005}
Nguyen Tien Zung, {\it Convergence versus integrability in Birkhoff normal form}, 
Ann. of Math. (2) 161 (2005), no. 1, 141–156.

\bibitem{Zung-Torus2006}
Nguyen Tien  Zung, \emph{Torus actions and integrable systems},  in Topological Methods in the Theory of Integrable Systems,
Editors A.V. Bolsinov, A.T. Fomenko and A.A. Oshemkov, Cambridge Scientific Publications (2006), 289--328.


\bibitem{ZungMinh_2D2012}
Nguyen Tien Zung, Nguyen Van Minh, {\it Geometry of integrable dynamical systems on 2-dimensional surfaces},
preprint arXiv:1204.1639 (2012).

\end{thebibliography}
\end{document}